\documentclass[12pt]{amsart}
\usepackage[T1]{fontenc}
\usepackage[latin1]{inputenc}
\usepackage{typearea}
\usepackage{geometry}
\usepackage{ulem}
\usepackage{amsmath}
\usepackage{amssymb}
\usepackage{latexsym}
\usepackage{enumerate}
\usepackage{amsthm}
\usepackage[all]{xy}
\usepackage{hhline}
\usepackage{epsf} %fuer bildchen
\usepackage{cite}
\usepackage{verbatim}
\usepackage{mathtools}
\usepackage{comment}
\usepackage{dsfont}
\usepackage{mathrsfs}

\usepackage{color}

\newtheorem{theorem}{{\sc Theorem}}[section]
\newtheorem{lemma}[theorem]{{\sc Lemma}}

\theoremstyle{remark}

\theoremstyle{definition}

\newtheorem{example}[theorem]{\sc example}

\newcommand{\R}{\mathbb{R} }
\newcommand{\N}{\mathbb{N} }

\newcommand{\B}{\mathcal{B}}
\newcommand{\F}{\mathcal{F}}

\newcommand{\Prob}{\mathbb{P}}

\newcommand{\Pot}{\mathcal{P}}
\newcommand{\Om}{\Omega}
\newcommand{\om}{\omega}

\newcommand{\void}{\emptyset}

\renewcommand{\phi}{\varphi}
\renewcommand{\epsilon}{\varepsilon}

\renewcommand{\rho}{\varrho}
\renewcommand{\P}{\Prob}

\begin{document}
\title[]{Discrete probability spaces revisited}
\author{Christian D\"obler}
\thanks{\noindent Mathematisches Institut der Heinrich Heine Universit\"{a}t D\"usseldorf\\
Email: christian.doebler@hhu.de }
\begin{abstract}  
We give an elementary proof of the known fact that every probability measure, defined on an arbitrary $\sigma$-field on a countable sample space $\Omega$, may in fact be extended to a probability measure on the power set of $\Om$. This result is further discussed and motivated in the context of discrete random variables.    
\end{abstract}

\maketitle

\section{Introduction}\label{intro}
In axiomatic probability theory, a probability space $(\Om,\F,\P)$ is customarily called \textit{discrete}, if 
\begin{enumerate}[(a)]
 \item the sample space $\Om$ is countable, i.e. either finite or countably infinite, and 
 \item $\F=\Pot(\Om)$ is the power set of $\Om$.
\end{enumerate}
 In this case, the function $p:\Om\rightarrow\R$ defined by $p(\om):=\P(\{\om\})$, $\om\in\Om$, is a \textit{probability mass function} (p.m.f.) on $\Om$, i.e.
\begin{enumerate}[(i)]
\item $p(\om)\geq0$ for all $\om\in\Om$ and 
\item $\sum_{\om\in\Om} p(\om)=1$.
\end{enumerate}
Conversely, if $\Om$ is a countable set and $p:\Om\rightarrow\R$ satisfies (i) and (ii) above, then the set function $\P:\Pot(\Om)\rightarrow\R$ with 
\begin{equation}\label{probabs}
 \P(A):=\sum_{\om\in A} p(\om),\quad A\subseteq\Om,
\end{equation}
is a probability measure on the measurable space $(\Om,\Pot(\Om))$, thus making it into a discrete probability space. 
This yields the well-known one-to-one correspondence between discrete probability spaces $(\Om,\Pot(\Om),\P)$ and pairs $(\Om,p)$ consisting of a countable set $\Om$ and a p.m.f. $p$ on $\Om$.

In particular, a probability space $(\Om,\F,\P)$ with countable sample space $\Omega$ only falls under the above definition of a discrete probability space, if $\F$ happens to be the power set $\Pot(\Om)$ of $\Om$. Hence, one might wonder, if each such probability measure on $\F\subsetneq\Pot(\Om)$ may always be extended to a probability measure $\P^*$ on $\Pot(\Om)$. In view of the above one-to-one correspondence, this is equivalent to asking, whether there is always a p.m.f. $p:\Om\rightarrow\R$ such that \eqref{probabs} holds for each $A\in\F$.\\
%\[\P(A)=\sum_{\om\in A}p(\om),\quad A\in\F.\]

Such a question in particular arises in the context of \textit{discrete random variables}. Recall that, if $(\Om,\F,\P)$ is now an arbitrary probability space, that is $\Om$ is not necessarily assumed countable, and $(E,\B)$ is another measurable space, then a random variable $X:(\Om,\F)\rightarrow(E,\B)$ (i.e. an $\F-\B$-measurable mapping) is called \textit{discrete}, whenever there is a countable set $D\in\B$ such that $\P(X\in D)=1$. Note that discrete random variables are often most naturally defined on uncountable sample spaces $\Omega$. One may, for instance, think of an infinite sequence of independent, fair coin tosses and let $X$ be the first time $n\in\N$ the coin shows heads, if any. Then, the natural sample space $\Om=\{0,1\}^\N$ is uncountable and $X$ has the geometric distribution on $D=\N=\{1,2,\dotsc\}$ with parameter $1/2$ and, hence, is a discrete random variable.

Since, for a discrete random variable $X$, the distribution $\P_X$ of $X$ on $(E,\B)$ is concentrated on $D$ we may view it as a probability measure on $(D,\B|D)$, where 
\[\B|D:=\{B\cap D\,:\,B\in\B\}\]
is the trace of $\B$ on $D$. 
Hence, when exploring distributional properties of a discrete random variable, we may w.l.o.g. assume that $E$ itself is countable and $D=E$. In this case, according to what has been noticed above, the probability space $(E,\B,\P_X)$ induced by a discrete random variable is \textbf{not} necessarily discrete in the above sense, since the $\sigma$-field $\B$ might be strictly smaller than the power set $\Pot(E)$. This means that probabilities of the form 
\begin{equation}\label{probs}
 \P(X=x),\quad x\in E,
\end{equation}
may not be well-defined for $\{x\}\notin\B$.

It should however be mentioned that this issue does not arise in the situation of discrete random variables 
$X:(\Om,\F)\rightarrow(\R^d,\B(\R^d))$, where $\B(\R^d)$ denotes the \text{Borel $\sigma$-field} on $\R^d$, since the singeltons $\{x\}$, $x\in\R^d$, are contained in $\B(\R^d)$. Hence, if $D\in \B(\R^d)$ is the given countable set belonging to $X$ such that $\P(X\in D)=1$, then it follows that $\B(\R^d)|D=\Pot(D)$. In particular, 
$(D,\B(\R^d)|D,\P_X)$ is in fact a discrete probability space in this situation.\\

In a converse sense, even if the probabilites \eqref{probs} are all well-defined for a given random variable $X$ which is in fact defined on a countable sample space $\Om$, the probability measure $\P$ on the underlying probability space might stick lack some information, as is illustrated by the following example.

\begin{example}\label{finance}
Suppose that the countable set $\Omega$ describes all possible scenarios in a financial market and that $X(\omega)$ is the value of some stock at a fixed time, if the scenario $\omega$ has occurred. Then, we might well be aware of all probabilities \eqref{probs} for the possible values $x$ of the stock in the countable set $E\subseteq[0,\infty)$, implying in particular that $\B=\Pot(E)$, but the given $\sigma$-field $\F$ on $\Om$ might be as small as
\[\F=\sigma(X)=\bigl\{X^{-1}(B)\,:\,B\subseteq E\bigr\}.\]
This in particluar means that the probability that a particular scenario $\omega\in\Omega$ occurs, might not be computable. 

Observe however that in this situation, the $\sigma$-field $\F$ is generated by the countable, measurable partition (see below for a precise definition)
\[A_x:=X^{-1}\bigl(\{x\}\bigr),\quad x\in E\cap X(\Om),\]
which makes it easy to extend the probability $\P$ from $\F$ to $\Pot(\Om)$ (see the proof of Theorem \ref{maintheo} below in Section \ref{main}). Thus, it is actually possible to assign probabilites 
\[p(\om)=\P\bigl(\{\omega\}\bigr)\]
to all possible market scenarios in such a way that 
\[\sum_{\om\in A_x}p(\om)=\P(X=x)\]
for all $x\in E$. Moreover, the sizes of the pairwise disjoint sets $A_x$, $x\in X$, determine the number of degrees of freedom for such a choice.
\end{example}

It is the goal of this note to give a simple proof of the fact that, for a countable sample space $\Om$, it is always possible to extend a probability measure $\P$ on any $\sigma$-field $\F$ on $\Omega$ to a probability measure $\P^*$ on $\Pot(\Om)$ as in the previous example. 

\begin{theorem}\label{maintheo}
Let $(\Om,\F,\P)$ be a probability space such that $\Om$ is countable. Then, there is an extension $\P^*$ of $\P$ to $\Pot(\Om)$. In other words, there is a p.m.f. $p:\Om\rightarrow\R$ such that 
$\P(A)=\sum_{\om\in A}p(\om)$ for all $A\in\F$.
\end{theorem}

Theorem \ref{maintheo} is actually a direct corollary of a classical result of Bierlein, see \cite[Satz 2B]{Bierlein} (in German), \cite[Corollary 2]{AL} or \cite[Theorem 1.12.15]{Bogachev}. The proofs in these references are however quite demanding and require advanced knowledge of measure theory, whereas our direct proof of Theorem \ref{maintheo} only makes use of basic measure theoretic facts. This makes it possible to include it into an introductory course of probability. 

\section{Proof of Theorem \ref{maintheo}}\label{main}
The proof of Theorem \ref{maintheo} relies on the fact that, as in the example above, every $\sigma$-field on a countable set $\Om$ is generated by a countable measurable partition of $\Om$. We begin by properly defining these notions.

If $(\Om,\F)$ is a measurable space, an \textit{$\F$-measurable partition} of $\Om$ is a collection $\{B_i\,:\, i\in I\}\subseteq \F\setminus\{\void\}$ such that $\Om=\bigcup_{i\in I} B_i$ and $B_i\cap B_j=\void$ for all $i,j\in I$ such that $i\not=j$. Such a collection is called \textit{finite} or \textit{countable}, if the index set $I$ is finite or countable, respectively. If $\{B_i\,:\, i\in I\}$ is an $\F$-measurable partition of $\Om$, then every $A\in\F$ can be (uniquely) written in the form 
\[A=\bigcup_{\substack{i\in I:\\ B_i\subseteq A}} B_i.\] 
In particular, if $I$ is countable, then all such unions are again contained in the $\sigma$-field $\F$, so that $\F$ is in fact generated by the partition $\{B_i\,:\, i\in I\}$, in this case. \\

The following known result is fundamental to our proof. Since I did not manage to find a suitable reference for it, a complete proof is included.
\begin{lemma}\label{partition}
If $(\Om,\F)$ is a measurable space such that $\Om$ is countable, then 
%In the situation of Theorem \ref{maintheo}, 
there is a countable $\F$-measurable partition $\{B_i\,:\, i\in I\}$ of $\Om$. % generating $\F$. 
If $\F$ is finite, then so is the partition.
\end{lemma} 

\begin{proof}
The idea of the proof is to take the minimal (with respect to ``$\subseteq$'') non-empty measurable sets in $\F$ as the members of the sought partition. Since $\Om$ might be countably infinite, some care is needed in order to properly identify these. 

For $\om\in\Om$ we define 
\[C_\om:=\bigcap_{\substack{A\in \F:\\ \om\in A}} A.\]
If $\F$ is not finite, then, as is well-known, it is uncountable so it is not a priori clear that the $C_\om\in\F$. To see that this is in fact the case, observe first that $C_\om$ consists precisely of those points $\eta\in\Om$ such that for all $A\in\F$:
\[\om\in A\Rightarrow\eta\in A.\]
Now, for all pairs $(\om,\eta)\in\Om\times\Om$, let $D_{\om,\eta}\in\F$ be such that $\om\in D_{\om,\eta}$ but $\eta\notin D_{\om,\eta}$, if any. If there is no such set $D_{\om,\eta}\in\F$, then let 
 $D_{\om,\eta}:=\Om$ so that in particular $D_{\om,\om}=\Om$ for all $\om\in\Om$. We claim that 
\begin{equation}\label{part1}
C_\om=\bigcap_{\eta\in\Om} D_{\om,\eta},\quad\om\in\Om.
\end{equation}
To see this, fix $\om\in\Om$ and denote the right hand side in $\eqref{part1}$ by $C_\om'$. Since $\Om$ is countable and all $D_{\om,\eta}\in\F$, we have $C_\om'\in\F$. Moreover, as $\om\in C_\om'\in\F$ it is clear that $C_\om\subseteq C_\om'$ by the definition of $C_\om$. Conversely, suppose that 
$\gamma\in C_\om'$. %If $\gamma=\om$, then $\gamma\in C_\om$. If $\gamma\not=\om$, 
Then, in particular, $\gamma\in D_{\om,\gamma}$ so that $D_{\om,\gamma}=\Om$ and, 
thus, every $A\in\F$ that contains $\om$ must also contain $\gamma$. Thus, $\gamma\in C_\om$ and \eqref{part1} is proved. %Since the intersection on the right hand side of \eqref{part1} is countable
Since the intersection in \eqref{part1} is countable, it in particular follows that $C_\om\in\F$ for all $\om\in\Om$.

Next, we claim that two such sets $C_\om$ and $C_{\om'}$, $\om,\om'\in\Om$, are either identical or disjoint. Suppose, for instance, that there is an $\eta\in C_{\om}\setminus C_{\om'}$. Then, each $A\in \F$ that contains $\om$ also contains $\eta$ but there is a $B\in\F$ such that $\om'\in B$ and $\eta\notin B$. Hence, $B$ cannot contain $\om$ either and, a fortiori, $\om\in C_{\om'}^c\in\F$. Thus, $C_{\om'}^c$ appears in the intersection defining $C_\om$, implying that $C_\om\cap C_{\om'}=\emptyset$. 

Now, if $I$ is a (necessarily countable) index set and $B_i$, $i\in I$, is a suitable enumeration of the pairwise distinct ones among the sets $C_\om$, $\om\in \Om$, then the desired partition is found. Indeed, if $A\in\F$, then 
\[A=\bigcup_{\om\in A} C_\om.\]
 If $\F$ is finite, then $I$ must necessarily be finite as well, since the mapping 
\[\Pot(I)\ni J\mapsto\bigcup_{i\in J} B_i\in\F\]
is injective, as the $B_i$ are disjoint and non-empty.
\end{proof}

\begin{proof}[Proof of Theorem \ref{maintheo}]
By Lemma \ref{partition} there is a countable $\F$-measurable partition $\{B_i\,:\, i\in I\}$ generating $\F$. Fix $i\in I$. If $B_i$ is finite, let 
\[p(\om):=\frac{\P(B_i)}{|B_i|},\quad \om\in B_i.\]
If $B_i$ is countably infinite, then let $\phi_i:B_i\rightarrow\N$ be any bijection and define
\[p(\om):=\frac{\P(B_i)}{2^{\phi_i(\om)}},\quad \om\in B_i.\]
Then, $p(\om)\geq0$ for all $\om\in\Om=\bigcup_{i\in I}B_i$ and 
\[\sum_{\om\in B_i}p(\om)=\P(B_i)\]
for each $i\in I$. Moreover, each $A\in \F$ can be written as the disjoint union
\[A=\bigcup_{\substack{i\in I:\\ B_i\subseteq A}} B_i\] 
so that
\begin{align*}
\sum_{\om\in A} p(\om)&=\sum_{\substack{i\in I:\\ B_i\subseteq A}}\Bigl(\sum_{\om\in B_i} p(\om)\Bigr)=\sum_{\substack{i\in I:\\ B_i\subseteq A}} \P(B_i)=\P(A).
\end{align*}
In particular, by letting $A=\Om$ in the last display, it follows that $p$ is a p.m.f., concluding the proof.
\end{proof}

By inspection of the above proof we can further infer the following about the number of possible extensions in the situation of Theorem \ref{maintheo}: If $\{B_i\,:\, i\in I\}$ is a fixed, countable $\F$-measurable partition generating $\F$, then there is a one-to-one correspondence $(q_i)_{i\in I}\mapsto p$ between families of p.m.f.'s $q_i$ on $B_i$, $i\in I$, and p.m.f.'s $p$ on 
$\Omega$ such that $\sum_{\om\in A}p(\om)=\P(A)$ for all $A\in\F$.

\normalem
\bibliography{disprob}{}
\bibliographystyle{alpha}
\end{document}